\documentclass[12pt,a4paper,leqno]{amsart}

\usepackage[latin1]{inputenc}
\usepackage[T1]{fontenc}
\usepackage{amsfonts}
\usepackage{amsmath}
\usepackage{amssymb}
\usepackage{eurosym}
\usepackage{mathrsfs}
\usepackage{palatino}
\usepackage{color}
\usepackage{esint}
\usepackage{url}

\newcommand{\R}{\mathbb{R}}
\newcommand{\C}{\mathbb{C}}

\newcommand{\N}{\mathbb{N}}

\newcommand{\Z}{\mathbb{Z}}

\numberwithin{equation}{section}

\swapnumbers
\theoremstyle{plain}
\newtheorem{thm}[equation]{Theorem}
\newtheorem{lem}[equation]{Lemma}
\newtheorem{prop}[equation]{Proposition}

\theoremstyle{definition}

\newtheorem{ass}[equation]{Assumption}

\theoremstyle{remark}
\newtheorem{rem}[equation]{Remark}

\title[Bi-parameter square functions in the upper half-space]{Boundedness of a class of bi-parameter square functions in the upper half-space}

\author{Henri Martikainen}
\address{D\'epartement de Math\'ematiques, B\^atiment 425, Facult\'e des Sciences d'Orsay, Universit\'e Paris-Sud 11, F-91405 Orsay Cedex}
\email{henri.martikainen@math.u-psud.fr}
\thanks{The author is supported by the Emil Aaltonen Foundation, and wishes to thank Universit\'e Paris-Sud 11, Orsay, for its hospitality.}

\makeatletter
\@namedef{subjclassname@2010}{%
  \textup{2010} Mathematics Subject Classification}
\makeatother


\begin{document}
\maketitle

\begin{abstract}
We consider a class of bi-parameter kernels and related square functions in the upper half-space, and give an efficient proof of a boundedness criterion for them. The proof
uses modern probabilistic averaging methods and is based on controlling double Whitney averages over good cubes.
\end{abstract}

\section{Introduction}
In this paper we introduce a class of bi-parameter kernels $(t_1, t_2, x_1, x_2, y_1, y_2) \mapsto s_{t_1, t_2}(x_1, x_2, y_1, y_2)$, where $t_1, t_2 > 0$ and $x = (x_1, x_2), y = (y_1, y_2) \in \R^{n+m}$.
These kernels are assumed to satisfy a natural size estimate, a H\"older estimate and two symmetric mixed H\"older and size estimates. We also assume certain mixed Carleson and size estimates, mixed Carleson and H\"older estimates and a bi-parameter Carleson condition. Under these conditions
we show the square function bound
\begin{displaymath}
\iint_{\R^{m+1}_+} \iint_{\R^{n+1}_+} |\theta_{t_1, t_2} f(x_1, x_2)|^2 \frac{dx_1dt_1}{t_1}\frac{dx_2dt_2}{t_2} \lesssim \|f\|_{L^2(\R^{n+m})}^2,
\end{displaymath}
where
\begin{displaymath}
\theta_{t_1, t_2} f(x_1,x_2) = \iint_{\R^{n+m}} s_{t_1, t_2}(x_1, x_2, y_1, y_2)f(y_1, y_2)\,dy_1\,dy_2.
\end{displaymath}

Compared to the bi-parameter Calder\'on--Zygmund theory the square function case is significantly cleaner. Indeed, the amount of needed symmetries and conditions are greatly reduced. Moreover, one
encounters only one \emph{full} paraproduct -- not four. In particular, some demanding aspects related to mixed full paraproducts arising from partial adjoints of Calder\'on--Zygmund operators
are not present here.

Recently the author together with M. Mourgoglou \cite{MM} proved a boundedness criterion for one-parameter square functions with general measures. The key to the short proof is
based on a new averaging identity over good Whitney regions. The identity is a further development of Hyt\"onen's improvement \cite{Hy} of the Nazarov--Treil--Volberg method of random dyadic systems \cite{NTV}.
In this paper we push this efficient proof strategy to the case of two parameters. Probabilistic methods in the bi-parameter Calder\'on--Zygmund setting were first used by the author in \cite{Ma1}. They saw another application
in a joint work with Hyt\"onen \cite{HM}. Even in the probabilistic realm the square function case is cleaner than the corresponding Calder\'on--Zygmund case.

The first $T1$ type theorem for product spaces was proved by Journ\'e \cite{Jo2}. Journ\'e formulated his theorem in the language of vector-valued Calder\'on--Zygmund theory. S. Pott and P. Villarroya \cite{PV}
recently offered a new view -- an alternative framework avoiding the vector-valued assumptions. This ideology of mixing the various conditions (kernel estimates, BMO and weak boundedness property)
was also used in \cite{Ma1} and \cite{HM}. The current paper is a continuation of this theme but in the square function setting. For the corresponding one-parameter square function theory see e.g. the papers by
Christ--Journ\'e \cite{CJ}, Hofmann \cite{Ho} and Semmes \cite{S}.

\section{The main theorem}
\subsection{Bi-parameter square functions}
If $f\colon \R^{n+m} \to \C$, $x = (x_1, x_2) \in \R^{n+m}$ and $t_1, t_2 > 0$, we let
\begin{displaymath}
\theta_{t_1, t_2} f(x) = \iint_{\R^{n+m}} s_{t_1, t_2}(x_1, x_2, y_1, y_2)f(y_1, y_2)\,dy_1\,dy_2.
\end{displaymath}
We are interested in the square function estimate
\begin{displaymath}
\iint_{\R^{m+1}_+} \iint_{\R^{n+1}_+} |\theta_{t_1, t_2} f(x_1, x_2)|^2 \frac{dx_1dt_1}{t_1}\frac{dx_2dt_2}{t_2} \lesssim \|f\|_{L^2(\R^{n+m})}^2.
\end{displaymath}

\begin{ass}[Full standard estimates]
The kernel $s_{t_1, t_2}\colon \R^{n+m} \times \R^{n+m} \to \C$ is assumed to satisfy the \emph{size estimate}
\begin{displaymath}
|s_{t_1, t_2}(x, y)| \lesssim \frac{t_1^{\alpha}}{(t_1 + |x_1-y_1|)^{n+\alpha}} \frac{t_2^{\beta}}{(t_2 + |x_2-y_2|)^{m+\beta}}.
\end{displaymath}
We also assume the \emph{H\"older estimate}
\begin{align*}
|s_{t_1, t_2}(x, y) - s_{t_1, t_2}(x, (z_1, y_2)) - &s_{t_1, t_2}(x, (y_1, z_2)) + s_{t_1, t_2}(x,z)|\\ &\lesssim \frac{|y_1-z_1|^{\alpha}}{(t_1 + |x_1-y_1|)^{n+\alpha}} \frac{|y_2-z_2|^{\beta}}{(t_2 + |x_2-y_2|)^{m+\beta}}
\end{align*}
whenever $|y_1-z_1| < t_1/2$ and $|y_2-z_2| < t_2/2$. Finally, we assume the following two \emph{mixed H\"older and size estimates}:
\begin{displaymath}
|s_{t_1, t_2}(x, y) - s_{t_1, t_2}(x, (y_1, z_2))| \lesssim \frac{t_1^{\alpha}}{(t_1 + |x_1-y_1|)^{n+\alpha}} \frac{|y_2-z_2|^{\beta}}{(t_2 + |x_2-y_2|)^{m+\beta}}
\end{displaymath}
whenever $|y_2-z_2| < t_2/2$, and
\begin{displaymath}
|s_{t_1, t_2}(x, y) - s_{t_1, t_2}(x, (z_1, y_2))| \lesssim \frac{|y_1-z_1|^{\alpha}}{(t_1 + |x_1-y_1|)^{n+\alpha}}  \frac{t_2^{\beta}}{(t_2 + |x_2-y_2|)^{m+\beta}}
\end{displaymath}
whenever $|y_1-z_1| < t_1/2$.
\end{ass}

\begin{ass}[Carleson condition $\times$ standard estimates]
If $I \subset \R^n$ is a cube with sidelength $\ell(I)$, we define the associated Carleson box $\widehat I = I \times (0, \ell(I))$.
We assume the following combinations of Carleson and size conditions: For every cube $I \subset \R^n$ and $J \subset \R^m$ there holds that
\begin{displaymath}
\Big( \iint_{\widehat I}\, \Big| \int_I s_{t_1, t_2}(x_1, x_2, y_1, y_2)\,dy_1\Big|^2 \frac{dx_1dt_1}{t_1} \Big)^{1/2} \lesssim |I|^{1/2}  \frac{t_2^{\beta}}{(t_2 + |x_2-y_2|)^{m+\beta}}
\end{displaymath}
and
\begin{displaymath}
\Big( \iint_{\widehat J}\, \Big| \int_J s_{t_1, t_2}(x_1, x_2, y_1, y_2)\,dy_2\Big|^2 \frac{dx_2dt_2}{t_2} \Big)^{1/2} \lesssim \frac{t_1^{\alpha}}{(t_1 + |x_1-y_1|)^{n+\alpha}} |J|^{1/2}.
\end{displaymath}

We also assume the following combinations of Carleson and H\"older conditions: For every cube $I \subset \R^n$ and $J \subset \R^m$ there holds that
\begin{align*}
\Big( \iint_{\widehat I}\, \Big| \int_I [s_{t_1, t_2}(x_1, x_2, y_1, y_2)-s_{t_1, t_2}(x_1, x_2, y_1, z_2)]&\,dy_1\Big|^2 \frac{dx_1dt_1}{t_1} \Big)^{1/2} \\ &\lesssim |I|^{1/2} \frac{|y_2-z_2|^{\beta}}{(t_2 + |x_2-y_2|)^{m+\beta}}
\end{align*}
whenever $|y_2-z_2| < t_2/2$, and
\begin{align*}
\Big( \iint_{\widehat J}\, \Big| \int_J [s_{t_1, t_2}(x_1, x_2, y_1, y_2)-s_{t_1, t_2}(x_1, x_2, z_1, y_2)]&\,dy_2\Big|^2 \frac{dx_2dt_2}{t_2} \Big)^{1/2} \\ &\lesssim \frac{|y_1-z_1|^{\alpha}}{(t_1 + |x_1-y_1|)^{n+\alpha}} |J|^{1/2}
\end{align*}
whenever $|y_1-z_1| < t_1/2$.
\end{ass}

\begin{ass}[Bi-parameter Carleson condition]
Let $\mathcal{D} = \mathcal{D}_n \times \mathcal{D}_m$, where $\mathcal{D}_n$ is a dyadic grid in $\R^n$ and $\mathcal{D}_m$ is a dyadic grid in $\R^m$.
For $I \in \mathcal{D}_n$, let $W_I = I \times (\ell(I)/2, \ell(I))$ be the associated Whitney region. We define the numbers
\begin{displaymath}
C_{IJ}^{\mathcal{D}} = \iint_{W_J} \iint_{W_I} |\theta_{t_1, t_2} 1(x_1,x_2)|^2 \frac{dx_1dt_1}{t_1}\frac{dx_2dt_2}{t_2}.
\end{displaymath}
We assume the following Carleson condition: for every $\mathcal{D} = \mathcal{D}_n \times \mathcal{D}_m$ there holds that
\begin{displaymath}
\mathop{\sum_{I \times J \in \mathcal{D}}}_{I \times J \subset \Omega} C_{IJ}^{\mathcal{D}} \lesssim |\Omega|
\end{displaymath}
for all sets $\Omega \subset \R^{n+m}$ such that $|\Omega| < \infty$ and such that for every $x \in \Omega$ there
exists $I \times J \in \mathcal{D}$ so that $x \in I \times J \subset \Omega$.
\end{ass}
\begin{rem}
For the readers convenience the necessity of the bi-parameter Carleson condition is discussed in Appendix \ref{sec:nec}.
\end{rem}

We can now formulate our main theorem, which we aim to prove in an efficient way using the modern tools.

\begin{thm}\label{thm:main}
The square function estimate
\begin{displaymath}
\iint_{\R^{m+1}_+} \iint_{\R^{n+1}_+} |\theta_{t_1, t_2} f(x_1, x_2)|^2 \frac{dx_1dt_1}{t_1}\frac{dx_2dt_2}{t_2} \lesssim \|f\|_{L^2(\R^{n+m})}^2
\end{displaymath}
holds with a constant depending only on the Assumptions formulated above.
\end{thm}

\subsection{Strategy of the proof}
Let $w_n = (w_n^i)_{i \in \Z}$ and $w_m =  (w_m^j)_{j \in \Z}$, where $w_n^i \in \{0,1\}^n$ and $w_m^j \in \{0,1\}^m$. Let $\mathcal{D}^0_n$ and $\mathcal{D}^0_m$ be the standard dyadic grids on $\R^n$ and $\R^m$ respectively.
In $\R^n$ we define the new dyadic grid $\mathcal{D}_n = \{I + \sum_{i:\, 2^{-i} < \ell(I)} 2^{-i}w_n^i: \, I \in \mathcal{D}_n^0\} = \{I + w_n: \, I \in \mathcal{D}_n^0\}$, where we simply have defined
$I + w_n := I + \sum_{i:\, 2^{-i} < \ell(I)} 2^{-i}w_n^i$. The dyadic grid $\mathcal{D}_m$ in $\R^m$ is similarly defined. There is a natural product probability structure on $(\{0,1\}^n)^{\Z}$ and $(\{0,1\}^m)^{\Z}$. Therefore, we have the independent
random dyadic grids $\mathcal{D}_n$ and $\mathcal{D}_m$ in $\R^n$ and $\R^m$ respectively. Even if $n=m$ we use two independent grids.

A cube $I \in \mathcal{D}_n$ is called bad if there exists such a cube $\tilde I \in \mathcal{D}_n$ that $\ell(\tilde I) \ge 2^r \ell(I)$ and $d(I, \partial \tilde I) \le \ell(I)^{\gamma_n}\ell(\tilde I)^{1-\gamma_n}$.
Here $\gamma_n = \alpha/(2n + 2\alpha)$, where $\alpha > 0$ appears in the kernel estimates. One notes that
$\pi_{\textrm{good}}^n := \mathbb{P}_{w_n}(I + w_n \textrm{ is good})$ is independent of $I \in \mathcal{D}^0_n$. The parameter $r$ is a fixed constant so that $\pi_{\textrm{good}}^n, \pi_{\textrm{good}}^m > 0$.
Moreover, for a fixed $I \in \mathcal{D}^0_n$
the set $I + w_n$ depends on $w_n^i$ with $2^{-i} < \ell(I)$, while the goodness of $I + w_n$ depends on $w_n^i$ with $2^{-i} \ge \ell(I)$. These notions are thus independent.
Analogous definitions and remarks related to $\mathcal{D}_m$ hold.

Let $h_I$ be an $L^2$ normalized Haar function related to $I \in \mathcal{D}_n$.
This means that $h_I$, $I = I_1 \times \cdots \times I_n$, is one of the $2^n$ functions $h_I^{\eta}$, $\eta = (\eta_1, \ldots, \eta_n) \in \{0,1\}^n$, defined by
\begin{displaymath}
h_I^{\eta} = h_{I_1}^{\eta_1} \otimes \cdots \otimes h_{I_n}^{\eta_n}, 
\end{displaymath}
where $h_{I_i}^0 = |I_i|^{-1/2}1_{I_i}$ and $h_{I_i}^1 = |I_i|^{-1/2}(1_{I_{i, l}} - 1_{I_{i, r}})$ for every $i = 1, \ldots, n$. Here $I_{i,l}$ and $I_{i,r}$ are the left and right
halves of the interval $I_i$ respectively. If $\eta \ne 0$ the Haar function is cancellative: $\int h_I = 0$. The cancellative Haar functions form an orthonormal basis of $L^2(\R^n)$.
If $a \in L^2(\R^n)$ we may thus write $a = \sum_{I \in \mathcal{D}_n} \sum_{\eta \in \{0,1\}^n \setminus \{0\}} \langle a, h_I^{\eta} \rangle h_I^{\eta}$. We suppress the finite $\eta$ summation
and simply write $a = \sum_I \langle a, h_I\rangle h_I$. We may expand a function $f$ defined in $\R^{n+m}$ using the corresponding product basis:
\begin{displaymath}
f = \sum_{I, J} f_{IJ} h_{I \times J} := \sum_{I, J} \langle f, h_{I} \otimes h_{J} \rangle h_{I} \otimes h_{J}.
\end{displaymath}

Let always $I_1, I_2 \in \mathcal{D}_n$ and $J_1, J_2 \in \mathcal{D}_m$. Using independence we see that
\begin{align*}
&\iint_{\R^{m+1}_+} \iint_{\R^{n+1}_+} |\theta_{t_1, t_2} f(x)|^2 \frac{dx_1dt_1}{t_1}\frac{dx_2dt_2}{t_2} \\
&= \frac{1}{\pi_{\textrm{good}}^n\pi_{\textrm{good}}^m} E_{w_n, w_m} \sum_{I_2, J_2 \textup{ good}} \iint_{W_{J_2}} \iint_{W_{I_2}} |\theta_{t_1, t_2} f(x)|^2 \frac{dx_1dt_1}{t_1}\frac{dx_2dt_2}{t_2} \\
&= \frac{1}{\pi_{\textrm{good}}^n\pi_{\textrm{good}}^m} E_{w_n, w_m} \sum_{I_2, J_2 \textup{ good}} \iint_{W_{J_2}} \iint_{W_{I_2}} \Big| \sum_{I_1, J_1} f_{I_1J_1} \theta_{t_1, t_2}h_{I_1 \times J_1}(x)\Big|^2 \frac{dx_1dt_1}{t_1}\frac{dx_2dt_2}{t_2}.
\end{align*}
Notice carefully that the cubes $I_1, J_1$ are not restricted to good cubes.

We conclude that we can fix the dyadic grids $\mathcal{D}_n$ and $\mathcal{D}_m$ and focus on dominating the sum
\begin{align*}
S = \sum_{I_2, J_2 \textup{ good}} \iint_{W_{J_2}} \iint_{W_{I_2}} \Big| \sum_{I_1, J_1} f_{I_1J_1} \theta_{t_1, t_2}h_{I_1 \times J_1}(x_1, x_2)\Big|^2 \frac{dx_1dt_1}{t_1}\frac{dx_2dt_2}{t_2}.
\end{align*}
The first step is to split
\begin{displaymath}
S \lesssim S_{<, <} + S_{\ge, \ge} + S_{<, \ge} + S_{\ge, <},
\end{displaymath}
where
\begin{displaymath}
S_{<, <} = \sum_{I_2, J_2 \textup{ good}}  \iint_{W_{J_2}} \iint_{W_{I_2}} \Big| \mathop{\mathop{\sum_{I_1, J_1}}_{\ell(I_1) < \ell(I_2)}}_{\ell(J_1) < \ell(J_2)} f_{I_1J_1}
\theta_{t_1, t_2}h_{I_1 \times J_1}(x_1, x_2)\Big|^2 \frac{dx_1dt_1}{t_1}\frac{dx_2dt_2}{t_2}
\end{displaymath}
and so on. In the square function setting $<$ is much easier than $\ge$. Indeed, no further splitting is necessary in the $<$ summations. It is only in these summations
that one may utilise the various H\"older type estimates.

However, the terms with $\ge$ have
to be further dominated by three pieces. For example, we dominate
\begin{align*}
S_{\ge, <} \lesssim S_{\ge_{\textup{sep}}, <} + S_{\supset, <} + S_{\sim, <}, 
\end{align*}
where
\begin{align*}
S_{\ge_{\textup{sep}}, <} = \sum_{I_2, J_2 \textup{ good}} \iint_{W_{J_2}} \iint_{W_{I_2}}& \Big| \mathop{\sum_{I_1:\, \ell(I_1) \ge \ell(I_2)}}_{d(I_1, I_2) > \ell(I_2)^{\gamma_n}\ell(I_1)^{1-\gamma_n}} \\ &\times \sum_{J_1:\, \ell(J_1) < \ell(J_2)}
f_{I_1J_1} \theta_{t_1, t_2}h_{I_1 \times J_1}(x)\Big|^2 \frac{dx_1dt_1}{t_1}\frac{dx_2dt_2}{t_2},
\end{align*}
\begin{align*}
S_{\supset, <} = \sum_{I_2, J_2 \textup{ good}} \iint_{W_{J_2}} \iint_{W_{I_2}} \Big| \sum_{I_1:\, I_2 \subsetneq I_1} \sum_{J_1:\, \ell(J_1) < \ell(J_2)}
f_{I_1J_1}  \theta_{t_1, t_2}h_{I_1 \times J_1}(x)\Big|^2 \frac{dx_1dt_1}{t_1}\frac{dx_2dt_2}{t_2}
\end{align*}
and
\begin{align*}
S_{\sim, <} =  \mathop{\sum_{I_1, I_2:\, \ell(I_1) \sim \ell(I_2)}}_{d(I_1,I_2) \lesssim \ell(I_2)} \sum_{J_2} \iint_{W_{J_2}}
 \iint_{W_{I_2}}  \Big| \sum_{J_1:\, \ell(J_1) < \ell(J_2)} f_{I_1J_1}  \theta_{t_1, t_2}h_{I_1 \times J_1}(x)\Big|^2 \frac{dx_1dt_1}{t_1}\frac{dx_2dt_2}{t_2}.
\end{align*}
The goodness was already used in the last term to force the condition $\ell(I_1) \le 2^r\ell(I_2)$. Indeed, initially this term contains the case $I_1 = I_2$ and also those
$I_1$ for which $\ell(I_1) \ge \ell(I_2)$, $I_1 \cap I_2 = \emptyset$ and $d(I_1, I_2) \le \ell(I_2)^{\gamma_n}\ell(I_1)^{1-\gamma_n}$. But in the latter case one has to have
 $\ell(I_1) \le 2^r\ell(I_2)$ by the goodness of $I_2$, since otherwise one would have  $d(I_1, I_2) > \ell(I_2)^{\gamma_n}\ell(I_1)^{1-\gamma_n}$. To move the $I_1$ summation outside the square we also used the fact that
 given $I_2$ there are $\lesssim 1$ cubes $I_1$ with these properties.
 
Goodness will be used one more time in the summations with $\supset$. Elsewhere the goodness
may be discarded.

Naturally, the most difficult term is $S_{\ge, \ge}$, which has to be dominated by nine terms
\begin{displaymath}
S_{\ge, \ge} \lesssim S_{\ge_{\textup{sep}}, \ge_{\textup{sep}}} + S_{\ge_{\textup{sep}}, \supset} + S_{\ge_{\textup{sep}}, \sim} + S_{\supset, \ge_{\textup{sep}}} + S_{\supset, \supset} + S_{\supset, \sim} +
S_{\sim, \ge_{\textup{sep}}} +  S_{\sim, \supset} + S_{\sim, \sim}.
\end{displaymath}
Here things boil down to using the various size conditions instead of the H\"older conditions.
\section{The term $S_{<,<}$}
This term is almost trivial. The full H\"older estimate gives that if $\ell(I_1) < \ell(I_2)$ and $\ell(J_1) < \ell(J_2)$, then
\begin{displaymath}
| \theta_{t_1, t_2}h_{I_1 \times J_1}(x)| \lesssim A_{I_1I_2}|I_2|^{-1/2} \cdot A_{J_1J_2}|J_2|^{-1/2}, \qquad (x_1, t_1) \in W_{I_2}, \, (x_2, t_2) \in W_{J_2}.
\end{displaymath}
Here
\begin{displaymath}
A_{I_1I_2} = \frac{\ell(I_1)^{\alpha/2}\ell(I_2)^{\alpha/2}}{D(I_1, I_2)^{n+\alpha}}|I_1|^{1/2}|I_2|^{1/2},
\end{displaymath}
where $D(I_1, I_2) = \ell(I_1) + \ell(I_2) + d(I_1,I_2)$. Therefore, we have that
\begin{displaymath}
S_{<,<} \lesssim \sum_{J_2} \sum_{I_2} \Big[ \sum_{I_1} A_{I_1I_2} \sum_{J_1} A_{J_1J_2} |f_{I_1J_1}|\Big]^2 \lesssim \|f\|_{L^2(\R^{n+m})}^2,
\end{displaymath}
where we have applied the following well-known Proposition twice.
\begin{prop}\label{prop:Abound}
There holds that
\begin{displaymath}
\sum_{I_1, I_2} A_{I_1I_2} x_{I_1} y_{I_2} \lesssim \Big( \sum_{I_1} x_{I_1}^2 \Big)^{1/2}  \Big( \sum_{I_2} y_{I_2}^2 \Big)^{1/2}
\end{displaymath}
for $x_{I_1}, y_{I_2} \ge 0$. In particular, there holds that
\begin{displaymath}
\sum_{I_2} \Big[ \sum_{I_1} A_{I_1I_2} x_{I_1} \Big]^2 \lesssim \sum_{I_1} x_{I_1}^2.
\end{displaymath}
\end{prop}

\section{The term $S_{\ge,<}$}
\subsubsection*{Term $S_{\ge_{\textup{sep}}, <}$}
The mixed H\"older and size estimate gives that in this case
\begin{displaymath}
| \theta_{t_1, t_2}h_{I_1 \times J_1}(x)| \lesssim \frac{\ell(I_2)^{\alpha}}{d(I_1,I_2)^{n+\alpha}} |I_1|^{1/2}\cdot A_{J_1J_2}|J_2|^{-1/2}, \qquad (x_1, t_1) \in W_{I_2}, \, (x_2, t_2) \in W_{J_2}.
\end{displaymath}
If $d(I_1,I_2) \ge \ell(I_1)$, then $D(I_1, I_2) \lesssim d(I_1, I_2)$. In this case, one has
\begin{displaymath}
\frac{\ell(I_2)^{\alpha}}{d(I_1,I_2)^{n+\alpha}} |I_1|^{1/2} \lesssim A_{I_1I_2}|I_2|^{-1/2}.
\end{displaymath}
If $d(I_1, I_2) \le \ell(I_1)$, then $D(I_1, I_2) \lesssim \ell(I_1)$. The condition $d(I_1, I_2) > \ell(I_2)^{\gamma_n}\ell(I_1)^{1-\gamma_n}$ together with the identity $\gamma_n(n+\alpha) = \alpha/2$ gives that
also in this case
\begin{displaymath}
\frac{\ell(I_2)^{\alpha}}{d(I_1,I_2)^{n+\alpha}} |I_1|^{1/2}  \lesssim \frac{\ell(I_1)^{\alpha/2}\ell(I_2)^{\alpha/2}}{\ell(I_1)^{n+\alpha}}|I_1|^{1/2} \lesssim A_{I_1I_2}|I_2|^{-1/2}.
\end{displaymath}
Therefore, we may conclude that $S_{\ge_{\textup{sep}}, <} \lesssim \|f\|_{L^2(\R^{n+m})}^2$ using the same argument as with the term $S_{<,<}$.

\subsubsection*{Term $S_{\supset, <}$}
Define
\begin{displaymath}
s^k_I = -1_{(I^{(k-1)})^c}  \langle h_{I^{(k)}}\rangle_{I^{(k-1)}} + \mathop{\sum_{\tilde I \in \textup{ch}(I^{(k)})}}_{\tilde I \ne I^{(k-1)}} 1_{\tilde I} h_{I^{(k)}}
\end{displaymath}
so that
\begin{displaymath}
h_{I^{(k)}} = s^k_I + \langle h_{I^{(k)}}\rangle_{I^{(k-1)}}.
\end{displaymath}
The thing to note about $s^k_I$ is that spt$\,s^k_I \subset (I^{(k-1)})^c$ and $|s^k_I| \lesssim |I^{(k)}|^{-1/2}$.
We also denote $f_{J_1} = \langle f, h_{J_1}\rangle$ so that $f_{J_1}(y_1) = \int f(y_1,y_2)h_{J_1}(y_2)\,dy_2$, $y_1 \in \R^n$. 

We now estimate
\begin{displaymath}
S_{\supset, <} \lesssim S_{\supset_{\textup{mod}}, <} + S_{\textup{Car}, <},
\end{displaymath}
where
\begin{displaymath}
S_{\supset_{\textup{mod}}, <} =  \sum_{I, J_2 \textup{ good}} \iint_{W_{J_2}} \iint_{W_{I}} \Big| \sum_{k=1}^{\infty} \sum_{J_1:\, \ell(J_1) < \ell(J_2)} f_{I^{(k)}J_1} \theta_{t_1, t_2}(s_I^k \otimes h_{J_1})(x)\Big|^2 \frac{dx_1dt_1}{t_1}\frac{dx_2dt_2}{t_2}
\end{displaymath}
and
\begin{align*}
S_{\textup{Car}, <} =  \sum_{I, J_2} \iint_{W_{J_2}} \iint_{W_{I}} \Big| \sum_{J_1:\, \ell(J_1) < \ell(J_2)} &  \theta_{t_1, t_2}(1 \otimes h_{J_1})(x) \\ &\times \sum_{k=1}^{\infty} \langle \Delta_{I^{(k)}} f_{J_1} \rangle_{I^{(k-1)}}
\Big|^2 \frac{dx_1dt_1}{t_1}\frac{dx_2dt_2}{t_2}.
\end{align*}
Note that $S_{\textup{Car}, <}$ collapses to
\begin{displaymath}
S_{\textup{Car}, <} = \sum_{I, J_2} \iint_{W_{J_2}} \iint_{W_{I}} \Big| \sum_{J_1:\, \ell(J_1) < \ell(J_2)} \langle f_{J_1} \rangle_I   \theta_{t_1, t_2}(1 \otimes h_{J_1})(x)  \Big|^2 \frac{dx_1dt_1}{t_1}\frac{dx_2dt_2}{t_2}.
\end{displaymath}

If $k > r$ we have by the goodness of $I$ that
\begin{align*}
\ell(I)^{\alpha} \int_{(I^{(k-1)})^c} \frac{dy_1}{|y_1-x_1|^{n+\alpha}} \le \ell(I)^{\alpha}& \int_{B(x_1, d(I, (I^{(k-1)})^c))^c} \frac{dy_1}{|y_1-x_1|^{n+\alpha}} \\ &\lesssim \ell(I)^{\alpha} d(I, (I^{(k-1)})^c)^{-\alpha} \lesssim 2^{-\alpha k/2}.
\end{align*}
Using the mixed H\"older and size estimate this yields that for $k > r$ there holds that
\begin{displaymath}
|\theta_{t_1, t_2}(s_I^k \otimes h_{J_1})(x)| \lesssim 2^{-\alpha k/2}|I^{(k)}|^{-1/2}  A_{J_1J_2} |J_2|^{-1/2}, \qquad (x_1, t_1) \in W_I, \, (x_2, t_2) \in W_{J_2}.
\end{displaymath}
In the case $k \le r$ we use the estimate
\begin{displaymath}
\int \frac{\ell(I)^{\alpha}}{(\ell(I) + |x_1-y_1|)^{n+\alpha}}\,dy_1 \lesssim \ell(I)^{-n} |3I| + \ell(I)^{\alpha} \int_{I^c} \frac{dy_1}{|y_1-c_I|^{n+\alpha}} \lesssim 1 \sim 2^{-\alpha k/2}
\end{displaymath}
to arrive at the same conclusion.
Therefore, we have that
\begin{align*}
 S_{\supset_{\textup{mod}}, <} &\lesssim \sum_k 2^{-\alpha k/2} \sum_I \frac{|I|}{|I^{(k)}|} \sum_{J_2} \Big[ \sum_{J_1:\, \ell(J_1) < \ell(J_2)} A_{J_1J_2}|f_{I^{(k)}J_1}| \Big]^2 \\
 &\lesssim  \sum_k 2^{-\alpha k/2} \sum_{U, J_1} \frac{|f_{UJ_1}|^2}{|U|} \sum_{I:\, I^{(k)} = U} |I| = \sum_k 2^{-\alpha k/2} \sum_{U, J_1} |f_{UJ_1}|^2 \lesssim \|f\|_{L^2(\R^{n+m})}^2.
\end{align*}

We then deal with $S_{\textup{Car}, <}$. Minkowski's integral inequality yields that $S_{\textup{Car}, <}$ can be dominated by
\begin{displaymath}
\sum_{J_2} \iint_{W_{J_2}} \Big[ \sum_{J_1:\, \ell(J_1) < \ell(J_2)} \Big( \sum_I |\langle f_{J_1} \rangle_I|^2 \iint_{W_I} |\theta_{t_1, t_2}(1 \otimes h_{J_1})(x)|^2 \frac{dx_1dt_1}{t_1} \Big)^{1/2} \Big]^2 \frac{dx_2dt_2}{t_2}.
\end{displaymath}
\begin{lem}\label{lem:Car1}
Let $J_1, J_2$ be such that $\ell(J_1) < \ell(J_2)$, and let $(x_2, t_2) \in W_{J_2}$.
The numbers
\begin{displaymath}
\iint_{W_I} |\theta_{t_1, t_2}(1 \otimes h_{J_1})(x)|^2 \frac{dx_1dt_1}{t_1}, \qquad I \in \mathcal{D}_n,
\end{displaymath}
satisfy the Carleson condition
\begin{align*}
\sum_{\tilde I \subset I} \iint_{W_{\tilde I}} |\theta_{t_1, t_2}(1 \otimes h_{J_1})(x)|^2 \frac{dx_1dt_1}{t_1} = \iint_{\widehat I} |\theta_{t_1, t_2}(1 \otimes &h_{J_1})(x)|^2 \frac{dx_1dt_1}{t_1} \\ &\lesssim (A_{J_1J_2}|J_2|^{-1/2})^2 |I|.
\end{align*}
\end{lem}
\begin{proof}
Fix the cube $I \in \mathcal{D}_n$. We estimate
\begin{align*}
\iint_{\widehat I} |\theta_{t_1, t_2}(1 \otimes h_{J_1})(x)|^2 \frac{dx_1dt_1}{t_1} \lesssim \iint_{\widehat {3I}} |\theta_{t_1, t_2}&(1_{3I} \otimes h_{J_1})(x)|^2 \frac{dx_1dt_1}{t_1} \\
&+ \iint_{\widehat I} |\theta_{t_1, t_2}(1_{(3I)^c} \otimes h_{J_1})(x)|^2 \frac{dx_1dt_1}{t_1}.
\end{align*}
Minkowski's integral inequality yields that
\begin{align*}
\iint_{\widehat {3I}}  &|\theta_{t_1, t_2}(1_{3I} \otimes h_{J_1})(x)|^2 \frac{dx_1dt_1}{t_1} \\ &\le |J_1|^{-1} \Big[ \int_{J_1} \Big( \iint_{\widehat {3I}} \Big| \int_{3I} [s_{t_1,t_2}(x,y) - s_{t_1,t_2}(x, (y_1, c_{J_1}))] \,dy_1 \Big|^2
\frac{dx_1dt_1}{t_1} \Big)^{1/2} dy_2 \Big)^2 \\
&\lesssim (A_{J_1J_2}|J_2|^{-1/2})^2 |I|,
\end{align*}
where the last estimate follows from the mixed Carleson and H\"older conditions.

The mixed H\"older and size estimate gives that
\begin{displaymath}
|\theta_{t_1, t_2}(1_{(3I)^c} \otimes h_{J_1})(x)| \lesssim t_1^{\alpha}\int_{I^c} \frac{dy_1}{|y_1-c_I|^{n+\alpha}} \cdot A_{J_1J_2}|J_2|^{-1/2} \lesssim t_1^{\alpha}\ell(I)^{-\alpha} \cdot A_{J_1J_2}|J_2|^{-1/2}
\end{displaymath}
yielding
\begin{align*}
\iint_{\widehat I} |\theta_{t_1, t_2}(1_{(3I)^c} \otimes h_{J_1})(x)|^2 \frac{dx_1dt_1}{t_1} \lesssim (A_{J_1J_2}|J_2|^{-1/2})^2 \cdot |I| \ell(I)&^{-2\alpha} \int_0^{\ell(I)} t_1^{2\alpha-1}\,dt_1 \\
&\lesssim  (A_{J_1J_2}|J_2|^{-1/2})^2 |I|.
\end{align*}
This completes the proof of the Lemma.
\end{proof}
Using the above lemma we now see that
\begin{align*}
S_{\textup{Car}, <} \lesssim \sum_{J_2} \Big[ \sum_{J_1:\, \ell(J_1) < \ell(J_2)}  A_{J_1J_2} \|f_{J_1}\|_{L^2(\R^n)} \Big]^2 \lesssim \sum_{J_1} \|f_{J_1}\|_{L^2(\R^n)}^2 \lesssim \|f\|_{L^2(\R^{n+m})}^2.
\end{align*}

\subsubsection*{Term $S_{\sim, <}$}
Here we have using the mixed H\"older and size estimate that for $(x_1, t_1) \in W_{I_2}, \, (x_2, t_2) \in W_{J_2}$ there holds that
\begin{displaymath}
| \theta_{t_1, t_2}h_{I_1 \times J_1}(x)| \lesssim  |I_1|^{1/2}\ell(I_2)^{-n}\cdot A_{J_1J_2}|J_2|^{-1/2} \sim |I_2|^{-1/2}\cdot A_{J_1J_2}|J_2|^{-1/2}.
\end{displaymath}
The last step uses the fact that here $\ell(I_1) \sim \ell(I_2)$. Using the fact that given $I_1$ we have $\lesssim 1$ cubes $I_2$ such that
$\ell(I_1) \sim \ell(I_2)$ and $d(I_1, I_2) \lesssim \ell(I_1)$, we have that
\begin{displaymath}
S_{\sim, <} \lesssim \sum_{I_1} \sum_{J_2} \Big[ \sum_{J_1} A_{J_1J_2} |f_{I_1J_1}| \Big]^2 \lesssim \|f\|_{L^2(\R^{n+m})}^2.
\end{displaymath}
\section{The term $S_{<, \ge}$}
This term is completely symmetric with the term $S_{\ge,<}$, and therefore the symmetric mixed H\"older and size estimate
and the symmetric mixed Carleson and H\"older estimate yield the bound $S_{<, \ge} \lesssim \|f\|_{L^2(\R^{n+m})}^2$.
\section{The term $S_{\ge, \ge}$}
\subsubsection*{Term $S_{\supset, \supset}$}
We need to bound
\begin{displaymath}
\sum_{I, J \textup{ good}} \iint_{W_J} \iint_{W_I} \Big| \sum_{k=1}^{\infty} \sum_{l=1}^{\infty} f_{I^{(k)}J^{(l)}} \theta_{t_1,t_2}(h_{I^{(k)}} \otimes h_{J^{(l)}})(x)\Big|^2 \frac{dx_1dt_1}{t_1} \frac{dx_2dt_2}{t_2}.
\end{displaymath}
Splitting $h_{I^{(k)}} = s^k_I + \langle h_{I^{(k)}}\rangle_{I^{(k-1)}}$ and $h_{J^{(l)}} = s^l_J + \langle h_{J^{(l)}}\rangle_{J^{(l-1)}}$ we dominate
\begin{displaymath}
S_{\supset, \supset} \lesssim S_{\supset_{\textup{mod}}, \supset_{\textup{mod}}} + S_{\supset_{\textup{mod}}, \textup{Car}} + S_{\textup{Car}, \supset_{\textup{mod}}} + S_{\textup{Car}, \textup{Car}},
\end{displaymath}
where
\begin{displaymath}
S_{\supset_{\textup{mod}}, \supset_{\textup{mod}}} = \sum_{I, J \textup{ good}} \iint_{W_J} \iint_{W_I} \Big| \sum_{k=1}^{\infty} \sum_{l=1}^{\infty} f_{I^{(k)}J^{(l)}} \theta_{t_1,t_2}(s^k_I \otimes s^l_J)(x)\Big|^2,
 \frac{dx_1dt_1}{t_1} \frac{dx_2dt_2}{t_2}
\end{displaymath}
\begin{displaymath}
S_{\supset_{\textup{mod}}, \textup{Car}} = \sum_{I, J \textup{ good}} \iint_{W_J} \iint_{W_I} \Big| \sum_{k=1}^{\infty} \langle f_{I^{(k)}} \rangle_J \theta_{t_1,t_2}(s^k_I \otimes 1)(x)\Big|^2
 \frac{dx_1dt_1}{t_1} \frac{dx_2dt_2}{t_2},
\end{displaymath}
\begin{displaymath}
S_{\textup{Car}, \supset_{\textup{mod}}} = \sum_{I, J \textup{ good}} \iint_{W_J} \iint_{W_I} \Big| \sum_{l=1}^{\infty} \langle f_{J^{(l)}} \rangle_I \theta_{t_1,t_2}(1 \otimes s^l_J)(x)\Big|^2
 \frac{dx_1dt_1}{t_1} \frac{dx_2dt_2}{t_2}
\end{displaymath}
and
\begin{displaymath}
S_{\textup{Car}, \textup{Car}} = \sum_{I, J \textup{ good}} |\langle f \rangle_{I \times J}|^2 \iint_{W_J} \iint_{W_I} |\theta_{t_1, t_2}1(x)|^2  \frac{dx_1dt_1}{t_1} \frac{dx_2dt_2}{t_2}.
\end{displaymath}

We begin with the term $S_{\textup{Car}, \textup{Car}}$. We have using the bi-parameter Carleson condition that
\begin{align*}
S_{\textup{Car}, \textup{Car}} \le \sum_{I, J} |\langle f \rangle_{I \times J}|^2 C_{IJ}^{\mathcal{D}} &= 2 \int_0^{\infty} \mathop{\sum_{I,J}}_{|\langle f \rangle_{I \times J}| > t} C_{IJ}^{\mathcal{D}} t\,dt \\
&\lesssim \int_0^{\infty} \mathop{\sum_{I,J}}_{I \times J \subset \{M_{\mathcal{D}}f > t\}} C_{IJ}^{\mathcal{D}} t\,dt \\
&\lesssim \int_0^{\infty} |\{M_{\mathcal{D}}f > t\}|t\,dt \lesssim \|M_{\mathcal{D}}f\|_{L^2(\R^{n+m})}^2 \lesssim \|f\|_{L^2(\R^{n+m})}^2.
\end{align*}

The size estimate gives that
\begin{displaymath}
|\theta_{t_1,t_2}(s^k_I \otimes s^l_J)(x)| \lesssim 2^{-\alpha k/2} |I^{(k)}|^{-1/2} 2^{-\beta l/2} |J^{(l)}|^{-1/2}.
\end{displaymath}
Indeed, this can be seen using the same argument that was used to estimate the term $S_{\supset_{\textup{mod}}, <}$. Similarly, this
then leads to the bound
\begin{displaymath}
S_{\supset_{\textup{mod}}, \supset_{\textup{mod}}} \lesssim \sum_{k,\,l}   2^{-\alpha k/2} 2^{-\beta l/2} \sum_{U,\,L} |f_{UL}|^2 \frac{1}{|U|}\sum_{I:\, I^{(k)} = U} |I| \cdot \frac{1}{|L|} \sum_{J: \, J^{(l)} = L} |J| \lesssim \|f\|_{L^2(\R^{n+m})}^2.
\end{displaymath}

This leaves us to deal with $S_{\textup{Car}, \supset_{\textup{mod}}}$ (the term $S_{\supset_{\textup{mod}}, \textup{Car}}$ is handled symmetrically).
One begins by dominating $S_{\textup{Car}, \supset_{\textup{mod}}}$ with
\begin{displaymath}
\sum_{J \textup{ good}} \iint_{W_J} \Big[ \sum_{\ell=1}^{\infty} \Big( \sum_I |\langle f_{J^{(l)}} \rangle_I |^2 \iint_{W_I} |\theta_{t_1,t_2}(1 \otimes s^l_J)(x)|^2  \frac{dx_1dt_1}{t_1} \Big)^{1/2} \Big]^2 \frac{dx_2dt_2}{t_2}.
\end{displaymath}
\begin{lem}
Let $J \in \mathcal{D}_{m, \textup{ good}}$, $(x_2, t_2) \in W_J$ and $\ell \in \N$ be fixed. The numbers
\begin{displaymath}
\iint_{W_I} |\theta_{t_1,t_2}(1 \otimes s^l_J)(x)|^2  \frac{dx_1dt_1}{t_1}, \qquad I \in \mathcal{D}_n,
\end{displaymath}
satisfy the Carleson condition
\begin{align*}
\sum_{\tilde I \subset I} \iint_{W_{\tilde I}} |\theta_{t_1,t_2}(1 \otimes s^l_J)(x)|^2  \frac{dx_1dt_1}{t_1} = \iint_{\widehat I} |\theta_{t_1,t_2}(1 \otimes &s^l_J)(x)|^2  \frac{dx_1dt_1}{t_1} \\
&\lesssim 2^{-\beta \ell} |J^{(l)}|^{-1} |I|.
\end{align*}
\end{lem}
\begin{proof}
The proof follows the idea of the proof of Lemma \ref{lem:Car1}. The difference is that one uses the mixed Carleson and size estimate and the full size condition.
Of course, one also uses the fact that since $J$ is good we have the familiar estimate
\begin{displaymath}
\int_{(J^{(l-1)})^c} \frac{\ell(J)^{\beta}}{(\ell(J) + |x_2-y_2|)^{m+\beta}}\,dy_2 \lesssim 2^{-\beta l/2}.
\end{displaymath}
\end{proof}
We conclude that
\begin{displaymath}
S_{\textup{Car}, \supset_{\textup{mod}}} \lesssim \sum_l 2^{-\beta l/2} \sum_L \|f_L\|_{L^2(\R^n)}^2 \frac{1}{|L|} \sum_{J: \, J^{(l)} = L} |J| \lesssim \|f\|_{L^2(\R^{n+m})}^2.
\end{displaymath}
\subsubsection*{Rest of the terms}
The remaining terms contain no new philosophies -- they only constitute an easy mixture of the already used techniques. For this reason we only shortly indicate how they are bounded.

The term $S_{\ge_{\textup{sep}}, \ge_{\textup{sep}}}$ is estimated using the size estimate. The correct bound is established similarly like the $\R^n$ part
was estimated in $S_{\ge_{\textup{sep}}, <}$. The bound $A_{I_1I_2}A_{J_1J_2}$ is then summed like in $S_{<, <}$.

The term $S_{\sim, \sim}$ is estimated using the size estimate after which the summation is trivial (see $S_{\sim, <}$).

The term $S_{\supset, \sim}$ is first split into $S_{\supset_{\textup{mod}}, \sim}$ and $S_{\textup{Car}, \sim}$. The term $S_{\supset_{\textup{mod}}, \sim}$ is handled using the size estimate.
The term $S_{\textup{Car}, \sim}$ is handled using the mixed Carleson and size estimate and the size condition. The term $S_{\sim, \supset}$ is of course symmetric.

The term $S_{\supset, \ge_{\textup{sep}}}$ is very similar to $S_{\supset, <}$. One simply uses the size estimate and the mixed Carleson and size estimate, and then the techniques
from the estimation of $S_{\ge_{\textup{sep}}, <}$ to get the $A_{J_1J_2}$. Otherwise there is no difference. The term $S_{\ge_{\textup{sep}}, \supset}$ is symmetric.

Finally, the terms $S_{\ge_{\textup{sep}}, \sim}$ and $S_{\sim, \ge_{\textup{sep}}}$ are bounded using the size condition (and the techniques from the estimation of $S_{\ge_{\textup{sep}}, <}$).
This concludes our proof of Theorem \ref{thm:main}.

\appendix
\section{Necessity of the bi-parameter Carleson condition}\label{sec:nec}
For the readers convenience we prove here that in the case of model operators our formulation of the bi-parameter Carleson condition is necessary for the square function bound in $L^2(\R^{n+m})$.

Suppose $\theta^n_{t_1}$ has a kernel $s^n_{t_1}(x_1,y_1)$, $\theta^m_{t_2}$ has a kernel $s^m_{t_2}(x_2,y_2)$, $x_1, y_1 \in \R^n$, $x_2, y_2 \in \R^m$, $t_1, t_2 > 0$. We assume that these
satisfy the size condition. Moreover, we assume the corresponding $L^2$ square function bounds in $\R^n$ and $\R^m$.
Let $\theta_{t_1, t_2} := \theta_{t_1}^n \otimes \theta_{t_2}^m$ satisfy an $L^2$ square function bound in $\R^{n+m}$.

Let $\mathcal{D} = \mathcal{D}_n \times \mathcal{D}_m$, where $\mathcal{D}_n$ is a dyadic grid in $\R^n$ and $\mathcal{D}_m$ is a dyadic grid in $\R^m$.
Let $\Omega \subset \R^{n+m}$ be such a bounded set that for every $x \in \Omega$ there exists $I \times J \in \mathcal{D}$ so that $x \in I \times J \subset \Omega$.
We will show that
\begin{equation}\label{eq:nec}
\mathop{\sum_{I \times J \in \mathcal{D}}}_{I \times J \subset \Omega} \iint_{W_J} \iint_{W_I} |\theta_{t_1, t_2} 1(x_1,x_2)|^2 \frac{dx_1dt_1}{t_1}\frac{dx_2dt_2}{t_2} \lesssim |\Omega|.
\end{equation}

Let $M_{\mathcal{D}}$ denote the strong maximal function related to the grid $\mathcal{D}$ and let $M$ denote the strong maximal function. Define
$\tilde \Omega = \{M_{\mathcal{D}}1_{\Omega} > 1/2\}$ and $\widehat \Omega = \{M1_{\tilde \Omega} > c\}$ for a small enough dimensional constant $c = c(n,m)$.
Since $|\widehat \Omega| \lesssim |\tilde \Omega| \lesssim |\Omega|$, it is enough by the square function bound in $L^2(\R^{n+m})$ to show that
\begin{displaymath}
\mathop{\sum_{I \times J \in \mathcal{D}}}_{I \times J \subset \Omega} \iint_{W_J} \iint_{W_I} |\theta_{t_1, t_2} 1_{{\widehat \Omega}^c}(x_1,x_2)|^2 \frac{dx_1dt_1}{t_1}\frac{dx_2dt_2}{t_2} \lesssim |\Omega|.
\end{displaymath}

In the proof we will select a plethora of various maximal dyadic cubes. We will be slightly lax with the justification of their existence -- for this minor detail we refer to \cite{HM}.

For every $J \in \mathcal{D}_m$ we let $\mathcal{F}_J$ consist of the maximal $F \in \mathcal{D}_n$ for which $F \times J \subset \tilde \Omega$. Then we define $F_J := \bigcup_{F \in \mathcal{F}_J} 2F$.
We will separately bound
\begin{displaymath}
S_1 := \mathop{\sum_{I \times J \in \mathcal{D}}}_{I \times J \subset \Omega} \iint_{W_J} \iint_{W_I} |\theta_{t_1, t_2}(1_{{\widehat \Omega}^c}1_{F_J})(x_1,x_2)|^2 \frac{dx_1dt_1}{t_1}\frac{dx_2dt_2}{t_2}
\end{displaymath}
and
\begin{displaymath}
S_2 := \mathop{\sum_{I \times J \in \mathcal{D}}}_{I \times J \subset \Omega} \iint_{W_J} \iint_{W_I} |\theta_{t_1, t_2}(1_{{\widehat \Omega}^c}1_{F_J^c})(x_1,x_2)|^2 \frac{dx_1dt_1}{t_1}\frac{dx_2dt_2}{t_2}.
\end{displaymath}

We deal with $S_1$ first. To this end, first fix $J \in \mathcal{D}_m$ and $(x_2, t_2) \in W_J$. We may forget the condition $I \times J \subset \Omega$ in the sum $S_1$. Therefore, we need to first bound
\begin{displaymath}
S_J(x_2, t_2) := \iint_{\R^{n+1}_+} |\theta_{t_1, t_2}(1_{{\widehat \Omega}^c}1_{F_J})(x_1,x_2)|^2 \frac{dx_1dt_1}{t_1}.
\end{displaymath}
We write
\begin{displaymath}
\theta_{t_1, t_2}(1_{{\widehat \Omega}^c}1_{F_J})(x_1,x_2) = \int s^m_{t_2}(x_2, y_2) \theta^n_{t_1}(1_{{\widehat \Omega}^c}(\cdot, y_2)1_{F_J})(x_1)\,dy_2
\end{displaymath}
after which Minkowski's integral inequality yields that
\begin{align*}
S_J(x_2,t_2) &\le \Big[ \int |s^m_{t_2}(x_2, y_2)| \Big( \iint_{\R^{n+1}_+} |\theta^n_{t_1}(1_{{\widehat \Omega}^c}(\cdot, y_2)1_{F_J})(x_1)|^2 \frac{dx_1dt_1}{t_1}\Big)^{1/2}dy_2\Big]^2 \\
&\lesssim \Big[ \int \frac{\ell(J)^{\beta}}{(\ell(J) + |x_2-y_2|)^{m+\beta}} \|1_{{\widehat \Omega}^c}(\cdot, y_2)1_{F_J}\|_{L^2(\R^n)}\,dy_2\Big]^2 \\
&\lesssim \int \frac{\ell(J)^{\beta}}{(\ell(J) + |x_2-y_2|)^{m+\beta}} \|1_{{\widehat \Omega}^c}(\cdot, y_2)1_{F_J}\|_{L^2(\R^n)}^2\,dy_2 \\
&\lesssim \int \frac{\ell(J)^{\beta}}{|y_2-c_J|^{m+\beta}} \|1_{{\widehat \Omega}^c}(\cdot, y_2)1_{F_J}\|_{L^2(\R^n)}^2\,dy_2 \\
&= \int 1_{F_J}(y_1) \int \frac{\ell(J)^{\beta}}{|y_2-c_J|^{m+\beta}} 1_{{\widehat \Omega}^c}(y_1, y_2)\,dy_2\,dy_1.
\end{align*}

Next, we have that
\begin{align*}
\sum_J &\iint_{W_J} S_J(x_2,t_2) \frac{dx_2dt_2}{t_2} \\ &\lesssim \sum_J |J| \int 1_{F_J}(y_1) \int \frac{\ell(J)^{\beta}}{|y_2-c_J|^{m+\beta}} 1_{{\widehat \Omega}^c}(y_1, y_2)\,dy_2\,dy_1 \\
&= \iint \Big( \mathop{\sum_{J:\, x_2 \in J}}_{y_1 \in F_J}  \int \frac{\ell(J)^{\beta}}{|y_2-c_J|^{m+\beta}} 1_{{\widehat \Omega}^c}(y_1, y_2)\,dy_2 \Big) dy_1\,dx_2.
\end{align*}
Notice that the outer integral can be restricted to
\begin{displaymath}
\bigcup_J (F_J \times J) \subset \widehat \Omega.
\end{displaymath}
Therefore, it suffices to show the pointwise bound
\begin{displaymath}
\mathop{\sum_{J:\, x_2 \in J}}_{y_1 \in F_J}  \int \frac{\ell(J)^{\beta}}{|y_2-c_J|^{m+\beta}} 1_{{\widehat \Omega}^c}(y_1, y_2)\,dy_2 \lesssim 1, \,\, (y_1, x_2) \in \bigcup_J (F_J \times J).
\end{displaymath}

Fix $(y_1, x_2)$ and let $T = T(y_1, x_2)$ be the maximal dyadic cube $J \in \mathcal{D}_m$ such that $x_2 \in J $ and $y_1 \in F_J$. In particular, $y_1 \in 2F$
for some $F \in \mathcal{F}_T$. If there would hold $y_2 \in 2T$, then $(y_1, y_2) \in 2F \times 2T \subset \widehat \Omega$, since $F \times T \subset \tilde \Omega$. But $1_{{\widehat \Omega}^c}(y_1, y_2) \ne 0$ so
$y_2 \in (2T)^c$. Now we have that
\begin{align*}
\int \frac{\ell(J)^{\beta}}{|y_2-c_J|^{m+\beta}} 1_{{\widehat \Omega}^c}(y_1, y_2)1_{(2T)^c}(y_2) \,dy_2 &\lesssim \int_{T^c} \frac{\ell(J)^{\beta}}{|y_2-c_T|^{m+\beta}}\, dy_2 \\
&\lesssim \Big(\frac{\ell(J)}{\ell(T)}\Big)^{\beta}.
\end{align*}
Therefore, we may bound
\begin{displaymath}
\mathop{\sum_{J:\, x_2 \in J}}_{y_1 \in F_J}  \int \frac{\ell(J)^{\beta}}{|y_2-c_J|^{m+\beta}} 1_{{\widehat \Omega}^c}(y_1, y_2)\,dy_2 \lesssim \sum_{J:\, x_2 \in J \subset T} \Big(\frac{\ell(J)}{\ell(T)}\Big)^{\beta}
= \sum_{j=0}^{\infty} 2^{-\beta j} \lesssim 1.
\end{displaymath}
We have thus proved that
\begin{displaymath}
S_1 \lesssim \iint_{\widehat \Omega} dy_1\,dx_2 = |\widehat \Omega| \lesssim |\Omega|.
\end{displaymath}

We will then bound $S_2$. This time we first fix $I \in \mathcal{D}_n$. Let $\mathcal{G}_I$ consist of maximal $G \in \mathcal{D}_m$ for which
$I \times G \subset \Omega$, and $I_G \in \mathcal{D}_n$ be the maximal cube for which $I_G \supset I$ and
$I_G \times G \subset \tilde \Omega$. Let $(x_1, t_1) \in W_I$. We write
\begin{align*}
S_I(x_1, t_1) &:= \sum_{J:\, I \times J \subset \Omega} \iint_{W_J} |\theta_{t_1, t_2}(1_{{\widehat \Omega}^c}1_{F_J^c})(x_1,x_2)|^2 \frac{dx_2dt_2}{t_2} \\
&= \sum_{G \in \mathcal{G}_I} \sum_{J: \, J \subset G} \iint_{W_J} |\theta_{t_1, t_2}(1_{{\widehat \Omega}^c}1_{F_J^c})(x_1,x_2)|^2 \frac{dx_2dt_2}{t_2}.
\end{align*}
Here we have that $I_G \times J \subset I_G \times G \subset \tilde \Omega$. Therefore, there exists $F \in \mathcal{F}_J$ such that $I_G \subset F$. This means that
$2I_G \subset 2F \subset F_J$. We will use this via $1_{F_J^c}(y_1) \le 1_{(2I_G)^c}(y_1)$. 

We have that
\begin{align*}
|\theta_{t_1, t_2}(1_{{\widehat \Omega}^c}1_{F_J^c})(x_1,x_2)| &\le \int 1_{(2I_G)^c}(y_1) |s^n_{t_1}(x_1,y_1)| |\theta^m_{t_2}(1_{{\widehat \Omega}^c}(y_1, \cdot))(x_2)|\,dy_1.
\end{align*}
Using this we may estimate
\begin{align*}
S_I(&x_1, t_1)  \\ &\le \sum_{G \in \mathcal{G}_I} \iint_{\widehat G} \Big[ \int 1_{(2I_G)^c}(y_1) |s^n_{t_1}(x_1,y_1)| |\theta^m_{t_2}(1_{{\widehat \Omega}^c}(y_1, \cdot))(x_2)|\,dy_1\Big]^2  \frac{dx_2dt_2}{t_2} \\
&\le \Big[ \int  |s^n_{t_1}(x_1,y_1)| \Big( \sum_{G \in \mathcal{G}_I} 1_{(2I_G)^c}(y_1) \iint_{\widehat G} |\theta^m_{t_2}(1_{{\widehat \Omega}^c}(y_1, \cdot))(x_2)|^2  \frac{dx_2dt_2}{t_2}\Big)^{1/2}dy_1\Big]^2 \\
&\lesssim \Big[ \int \frac{\ell(I)^{\alpha}}{(\ell(I) + |x_1-y_1|)^{n+\alpha}} \Big( \sum_{G \in \mathcal{G}_I} 1_{(2I_G)^c}(y_1) |G| \Big)^{1/2}dy_1\Big]^2 \\
&\lesssim \sum_{G \in \mathcal{G}_I} |G| \int \frac{\ell(I)^{\alpha}}{(\ell(I) + |x_1-y_1|)^{n+\alpha}}  1_{(2I_G)^c}(y_1)\,dy_1 \\
&\lesssim \sum_{G \in \mathcal{G}_I} |G| \int_{I_G^c} \frac{\ell(I)^{\alpha}}{|y_1-c_{I_G}|^{n+\alpha}}\,dy_1 \\
&\lesssim \sum_{G \in \mathcal{G}_I} |G| \Big(\frac{\ell(I)}{\ell(I_G)}\Big)^{\alpha}.
\end{align*}

We now recall the following dyadic instance of Journ\'e's lemma:
\begin{thm}[Journ\'e's lemma]\label{thm:Journe}
If $\omega:\N\to\R_+$ is decreasing and $\sum_{k=0}^{\infty} w(k) < \infty$, then
\begin{equation*}
  \sum_{\substack{R\subset\Omega\\ \textup{2-maximal} }}\omega(\operatorname{emb}_1(R;\Omega))\times|R|
  \leq 2\sum_{k=0}^\infty\omega(k)\times|\Omega|.
\end{equation*}
A  dyadic rectangle $R=I\times J\subset\Omega$ is 2-maximal if $I\times\tilde{J}\not\subset\Omega$ for any dyadic $\tilde{J}\supsetneq J$, and
$\operatorname{emb}_1(R;\Omega):=\sup\{k:R^{(k,0)}\subset\tilde\Omega\}$, $R^{(k,0)} := I^{(k)} \times J$.
\end{thm}

Using this we may now bound
\begin{align*}
S_2 = \sum_I \iint_{W_I} S_I(x_1,t_1) \frac{dx_1dt_1}{t_1} &\lesssim \sum_I \sum_{G \in \mathcal{G}_I} |I \times G| \Big(\frac{\ell(I)}{\ell(I_G)}\Big)^{\alpha} \\
&\le \mathop{\sum_{I \times G \subset \Omega}}_{I \times G \textup{ 2-maximal}} 2^{-\alpha \operatorname{emb}_1(I \times G;\Omega)} |I\times G| \lesssim |\Omega|.
\end{align*}
We have shown \eqref{eq:nec}.

\end{document}